\documentclass[11pt,letterpaper,reqno]{amsart}
\usepackage{amsfonts}
\usepackage{amsmath}
\usepackage{amsthm, amscd}
\usepackage{mathtools}
\usepackage{amssymb}
\usepackage{mathrsfs}
\usepackage{graphicx}
\usepackage{caption}
\usepackage{subcaption}
\usepackage{float}
\usepackage{xypic}
\usepackage[abs]{overpic}
\usepackage[alphabetic]{amsrefs}
\usepackage{etex}
\usepackage{tikz-cd}
\usepackage{hyperref}
\hypersetup{
	colorlinks,
	linkcolor=[rgb]{0,0,0.7},
	urlcolor=[rgb]{0,0,0.4},
	citecolor=[rgb]{0.4,0.1,0}
}

\allowdisplaybreaks

\theoremstyle{plain}\newtheorem{Theorem}{Theorem}[section]
\theoremstyle{plain}
\theoremstyle{plain}\newtheorem{Lemma}[Theorem]{Lemma}
\theoremstyle{plain}\newtheorem{Definition}[Theorem]{Definition}
\theoremstyle{plain}\newtheorem{Proposition}[Theorem]{Proposition}
\theoremstyle{plain}
\theoremstyle{plain}
\theoremstyle{plain}\newtheorem*{Claim*}{Claim}
\theoremstyle{plain}\newtheorem*{Theorem*}{Theorem}
\theoremstyle{plain}\newtheorem*{Lemma*}{Lemma}
\theoremstyle{plain}

\theoremstyle{remark}\newtheorem{remark}[Theorem]{Remark}
\theoremstyle{remark}
\theoremstyle{remark}\newtheorem*{Notation*}{Notation}

\numberwithin{equation}{section}

\DeclareMathOperator{\Ima}{Im}

\DeclareMathOperator{\SO}{SO}

\DeclareMathOperator{\id}{id}

\DeclareMathOperator{\Ext}{Ext}

\DeclareMathOperator{\Homeo}{Homeo}

\DeclareMathOperator{\Wh}{Wh}

\author{Jianfeng Lin}
\address{Yau Mathematical Sciences Center, Tsinghua University, Beijing 100084, China}
\email{linjian5477@mail.tsinghua.edu.cn}
\author{Yi Xie}
\address{Beijing International Center for Mathematical Research, Peking University, Beijing 100871, China}
\email{yixie@pku.edu.cn}
\author{Boyu Zhang}
\address{Department of Mathematics, The University of Maryland at College Park, Maryland 20742, USA}
\email{bzh@umd.edu}

\title{Splitting spheres for $S^2$--links in $S^4$}

\begin{document}

\maketitle 

\begin{abstract}
Suppose $L\cup R$ is a two-component sphere link in $S^4$. We show that if the link admits a smooth splitting $3$-sphere, then there exist infinitely many smooth splitting $3$-spheres that are topologically non-isotopic. 
This generalizes a theorem of Tatsuoka from the unlink to arbitrary split links. 
In the course of the proof, we establish a general sufficient condition under which a connected sum of smooth $4$-manifolds admits infinitely many topologically non-isotopic splitting $3$-spheres.
This criterion may be of independent interest; in particular, it applies to all previously known examples of nonuniqueness of splitting $3$-spheres for positive-genus surface links.
\end{abstract}

\section{Introduction}

For a two-component split link $K_1\sqcup K_2\subset S^3$, the splitting
sphere is unique up to isotopy in the link complement. In contrast, recent
works by Hughes--Kim--Miller \cite{hughes2025non}, Tatsuoka
\cite{tatsuoka2026splitting}, and Kim--Nahm--Tatsuoka
\cite[Sections~3--4]{kim2026knotting} have shown that the analogous
uniqueness statement fails for certain families of surface links in
$S^4$.

In this paper, we prove that for \emph{every} smooth two-component sphere link, if it admits a smooth splitting $3$-sphere, then it
admits infinitely many smooth splitting $3$-spheres that are pairwise topologically non-isotopic. 
In the proof, we also establish a general sufficient condition for the nonuniqueness of splitting spheres in connected sums of $4$--manifolds, as stated in
Proposition \ref{intro_prop_topological_nontrivial_case} below. Although
our main theorem concerns sphere links, this criterion also applies to
all the positive-genus split surface links studied in \cite{hughes2025non} and
\cite[Sections~3--4]{kim2026knotting}.

Our main theorem is as follows.

\begin{Theorem}
\label{thm_main}
Let $L\cup R\subset S^4$ be a smooth two-component sphere link that admits a smooth splitting $3$-sphere.
Then there exists an infinite family
$\{\Sigma_i\}_{i\in\mathbb{N}}$ of smoothly embedded $3$-spheres in
$S^4\setminus (L\cup R)$ such that each $\Sigma_i$ separates $L$ from
$R$; moreover, if $i\neq j$, then there exists no topological ambient isotopy of $S^4$
 fixing a neighborhood of $L\cup R$ that takes
$\Sigma_i$ to $\Sigma_j$.
\end{Theorem}

Suppose $\Sigma$ is a smooth splitting $3$-sphere of $L\cup R$. Consider the closed $4$-manifolds obtained by cutting $S^4$ open along $\Sigma$ and gluing each boundary component with $D^4$. The resulting manifold has two connected components, and we let $S_L$ denote the component that contains $L$, and let $S_R$ denote the component that contains $R$. By the topological Schoenflies theorem, both $S_L$ and $S_R$ are smooth $4$--manifolds homeomorphic to $S^4$. We may view $L$, $R$ as sphere knots in $S_L$, $S_R$ respectively.

The proof of Theorem \ref{thm_main} is divided into two cases. In the case that $L$ and $R$ are topologically unknotted in $S_L$ and $S_R$ respectively, the result is deduced from
the following extension of Tatsuoka's theorem 
\cite{tatsuoka2026splitting}.

\begin{Proposition}
\label{intro_prop_topological_trivial_case}
Let $X_1$ and $X_2$ be smooth $4$-manifolds homeomorphic to
$S^1\times D^3$, and let $X=X_1\# X_2$ denote their (smooth) interior connected
sum. Then there exists an infinite family
$\{\Sigma_i\}_{i\in\mathbb{N}}$ of smoothly embedded $3$-spheres in $X$
such that, for each $i$, the closures of the two components of
$X\setminus\Sigma_i$ are diffeomorphic to
$X_1\setminus\mathring{D}^4$ and
$X_2\setminus\mathring{D}^4$, respectively. Moreover, the spheres
$\Sigma_i$ are pairwise inequivalent under topological ambient isotopy
of $X$ relative to $\partial X$.
\end{Proposition}

Here, we use $\mathring{D}^4$ to denote the interior of a smoothly embedded closed $4$-ball in a given smooth $4$-manifold. 

The proof of Proposition \ref{intro_prop_topological_trivial_case} follows the strategy of Tatsuoka's argument \cite{tatsuoka2026splitting}.
The main additional issue is that one needs to construct barbell diffeomorphisms compatible with the smooth structure on $X_1\# X_2$, while
retaining control, in the standard coordinates, of the corresponding
self-homeomorphisms of $S^1\times D^3$ that arise in Tatsuoka's
construction. 

Proposition \ref{intro_prop_topological_trivial_case} will be proved in Section \ref{sec_topological_trivial}.

\medskip

In the case that at least one component, say $L$, is topologically
nontrivial in $S_L$, Freedman's unknotting theorem and Swarup's
unknotting criterion imply that a parallel copy of $L$ is not
null-homotopic in $S^4\setminus L$. To treat this case, we establish the
following general nonuniqueness criterion for connected sums of
$4$-manifolds.

\begin{Proposition}
\label{intro_prop_topological_nontrivial_case}
Let $X_1$ and $X_2$ be connected, oriented, compact smooth $4$-manifolds with
nonempty boundary. Suppose that the interior of $X_1$ contains a smoothly
embedded $2$-sphere with trivial normal bundle whose inclusion into
$X_1$ is not null-homotopic. Suppose also that there exists an element
$a\in\pi_1(X_2)$ such that $a^2\neq 1$. Let
$X=X_1\# X_2$ denote their (smooth) interior connected sum. Then there exists an
infinite family $\{\Sigma_i\}_{i\in\mathbb{N}}$ of smoothly embedded
$3$-spheres in $X$ such that, for each $i$, the closures of the two
components of $X\setminus\Sigma_i$ are diffeomorphic to
$X_1\setminus\mathring{D}^4$ and
$X_2\setminus\mathring{D}^4$, respectively; moreover, the spheres
$\Sigma_i$ are pairwise inequivalent under topological ambient isotopy
of $X$ relative to $\partial X$.
\end{Proposition}

In fact, we will show that the spheres $\Sigma_i$ in Proposition \ref{intro_prop_topological_nontrivial_case} can be taken to be mutually non-homotopic in the sense of Definition \ref{defn_homotopy_S3}. The proof of Proposition \ref{intro_prop_topological_nontrivial_case} and the full argument deducing Theorem \ref{thm_main} from Proposition \ref{intro_prop_topological_nontrivial_case} when $L$ is topologically non-trivial will be given in Section \ref{sec_topological_nontrivial}.

\begin{remark}
    Additional constructions of non-isotopic $S^3$'s in $4$-manifolds are established in \cite{budney2019knotted, iida2024diffeomorphisms, konno2022exotic}. Related results on non-isotopic embedded $D^3$'s are given in \cite{watanabe2020theta}.
\end{remark}

\textbf{Acknowledgments:} 
We would like to express our sincere gratitude to David Gabai and Alison Tatsuoka for helpful discussions, and to Maggie Miller for earlier correspondence about Freedman's unknotting theorem. We thank Michael Albanese for explaining an argument by Daniel Ruberman to us on MathOverflow, which is used in the proof of Proposition \ref{prop_4mfd_pi2_torsion}. 

J. Lin is partially supported by  National Key R\&D Program of China 2025YFA1017500 and NSFC 12271281. 
Y. Xie is partially supported by NSFC 12341105.
B. Zhang is partially supported by NSF grants DMS-2540516, DMS-2405271, and a travel grant from the Simons Foundation.

\section{The topologically trivial case}
\label{sec_topological_trivial}

Recall that $S_L$, $S_R$ are smooth $4$-manifolds homeomorphic to $S^4$ defined below the statement of Theorem \ref{thm_main}.
This section proves Theorem \ref{thm_main} in the case when $L$ is topologically ambiently isotopic to a trivial knot in $S_L$, and $R$ is topologically ambiently isotopic to a trivial knot in $S_R$. Note that since $\pi_0\Homeo^+(S^4)$ is trivial, the topological isotopy class of an unknot in an oriented topological $S^4$ is well-defined.  The main step is the following proposition, which is a restatement of Proposition \ref{intro_prop_topological_trivial_case}. 

\begin{Proposition}
\label{prop_topological_trivial_case}Let $X_1$ and $X_2$ be smooth $4$-manifolds homeomorphic to
$S^1\times D^3$, and let $X=X_1\# X_2$ denote their (smooth) interior connected
sum. Then there exists an infinite family
$\{\Sigma_i\}_{i\in\mathbb{N}}$ of smoothly embedded $3$-spheres in $X=X_1\# X_2$
such that, for each $i$, the closures of the two components of
$X\setminus\Sigma_i$ are diffeomorphic to
$X_1\setminus\mathring{D}^4$ and
$X_2\setminus\mathring{D}^4$, respectively; moreover, the spheres
$\Sigma_i$ are pairwise inequivalent under topological ambient isotopy
of $X$ relative to $\partial X$.
\end{Proposition}

We start by setting up some notation.
Since the smooth structure of every $3$--manifold is unique, the collar neighborhood of $\partial X_1$ must be diffeomorphic to the collar neighborhood of $\partial(S^1\times D^3)$. Hence we may attach a $2$-handle to $X_1$ to obtain a smooth manifold homeomorphic to $D^4$, and then attach a $4$-handle to obtain a smooth manifold homeomorphic to $S^4$. We use $\tilde{X}_1$ to denote the union of $X_1$ with the $2$-handle, and use $\hat{X}_1$ to denote the union of $X_1$ with the $2$, $4$-handles. Similarly, let $\tilde{X}_2$ denote the smooth manifold homeomorphic to $D^4$ that is obtained by attaching a $2$-handle to $X_2$, and let $\hat X_2$ denote the smooth manifold homeomorphic to $S^4$ obtained by attaching a $4$-handle to $\tilde{X}_2$. Let $\tilde{X} = \tilde{X}_1\#\tilde{X}_2$, $\hat{X} = \hat{X}_1\#\hat{X}_2$. It is clear that $\tilde{X}$ is homeomorphic to $S^3\times I$, and $\hat{X}$ is homeomorphic to $S^4$.

By the $4$-dimensional topological annulus theorem (\cite{quinn1982ends}, see also Theorem 4.1 of \cite{friedl2025foundations}), if $M$ is homeomorphic to $S^4$ and $B\subset M$ is a collared closed $4$-ball, then $M\setminus \mathring{B}$ is homeomorphic to the closed $4$-ball.

We invoke the following theorem of Budney--Gabai \cite{budney2025automorphism}:

\begin{Theorem}[Budney--Gabai]
\label{thm_Budney_Gabai}
    There exists a diffeomorphism $f:S^1\times D^3\to S^1\times D^3$ such that the following properties hold.
    \begin{enumerate}
        \item $f$ equals the identity near $\partial(S^1\times D^3)$ and is homotopic to the identity relative to $\partial(S^1\times D^3)$.
        \item $f^k\neq \id$ in $\pi_0\Homeo_\partial(S^1\times D^3)/\pi_0\Homeo_\partial(D^4)$ for all non-zero integers $k$.
        \item There exists a positive integer $m$, such that the lifting of $f$ to the connected $m$-fold cover of $S^1\times D^3$ is smoothly isotopic to the identity relative to boundary. 
    \end{enumerate}
\end{Theorem}

Here, if $X$ is a compact $4$-manifold, we use $\pi_0\Homeo_\partial(X)/\pi_0\Homeo_\partial(D^4)$ to denote the group of isotopy classes of self-homeomorphisms of  $X$ relative to the boundary, modulo the homeomorphisms that are supported in the interior of a topological $D^4$. For the statement of Property (3), note that if $M$ is a connected manifold with a non-empty boundary and $g$ is a self-homeomorphism of $M$ that is homotopic to the identity relative to $\partial M$, and if $M^\dagger$ is a covering space of $M$, then there is a unique self-homeomorphism $g^\dagger$ of $M^\dagger$ that fixes $\partial M^\dagger$ and lifts $g$. We call $g^{\dagger}$ the \emph{lifting} of $g$ to $M^\dagger$. 

\begin{remark}
    We briefly review how results in \cite{budney2025automorphism} yield Theorem \ref{thm_Budney_Gabai}. The map $f$ can be taken to be one of the $\delta_k$ diffeomorphisms in \cite[Section 3]{budney2025automorphism} with $k\ge 4$. Property (1) follows from the definition of $f$ and \cite[Proposition 2.6]{budney2025automorphism}. Property (2) follows from \cite[Section 4]{budney2025automorphism}. Property (3) holds for all sufficiently large $m$, and it follows from tracking the lifts of barbells and using \cite[Proposition 2.5]{budney2025automorphism}; this property is stated in the proof of \cite[Theorem 5.1]{budney2025automorphism}.
\end{remark}

For the rest of this section, let $f$ be a diffeomorphism given by Theorem \ref{thm_Budney_Gabai}.
If $\gamma: S^1\to X$ is a smooth framed embedding, then $\gamma$ extends to a smooth embedding of $S^1\times D^3$ in $X$, and the push-forward of $f$ defines a diffeomorphism on $X$. We denote the resulting diffeomorphism on $X$ by $f_\gamma$, and call it the \emph{implementation of $f$ by $\gamma$}.  The map $f_\gamma$ is well-defined up to smooth isotopy.

\begin{remark}
\label{rmk_implement_in_S1xD3}
Note that if we use $\gamma_0:S^1\to S^1\times D^3$ to denote the embedding of the core circle, then there are two framings of $\gamma_0$ up to homotopy, which yield two implementations of $f$ in $S^1\times D^3$. One implementation is isotopic to $f$, and the other one is isotopic to the conjugation of $f$ by a twist of $S^1\times D^3$ using the non-trivial element of $\pi_1\SO(4)$. Therefore, both implementations satisfy Properties (1) -- (3) of Theorem \ref{thm_Budney_Gabai}.
\end{remark}

Given an oriented embedded circle in $\tilde{X} = \tilde{X}_1\#\tilde{X}_2$, we define a canonical choice of normal framing on the circle which is invariant under isotopy. Since $\tilde{X}$ is homotopy equivalent to $S^3$, the classification of $\SO(4)$ bundles over $\tilde{X}$ is given by $\pi_2\SO(4)=0$, so the tangent bundle of $\tilde{X}$ is trivial.  
Fix a trivialization of $T\tilde{X}$ that is compatible with the orientation.
For each oriented embedded circle $\gamma:S^1\to \tilde{X}$, a trivialization of its normal bundle defines a loop in $\operatorname{O}(T\tilde{X})$ whose projection to $\tilde{X}$ is $\gamma$. Using the given trivialization of $T\tilde{X}$, such a loop projects to a loop in $\operatorname{O}(4)$. Since $\pi_1\SO(3)\cong \pi_1\SO(4)$,  there is a unique normal framing of $\gamma$ (up to homotopy) such that the corresponding loop is a contractible loop in $\SO(4)$. We define this to be the \emph{canonical framing} of $\gamma$. If $\gamma_1,\gamma_2$ are two oriented embedded circles that are smoothly isotopic to each other in $\tilde{X}$, then the isotopy takes the canonical framing of $\gamma_1$ to the canonical framing of $\gamma_2$ (up to homotopy). 

We have $\pi_1(X)\cong \mathbb{Z}*\mathbb{Z}$, where each $\mathbb{Z}$ factor is given by $\pi_1(X_i)$ ($i=1,2$). Take the base point of $X$ to be a fixed point on the standard splitting $S^3$ of $X$.
Take a smoothly embedded oriented circle $\gamma:S^1\to X$ that represents the element $1*1\in \mathbb{Z}*\mathbb{Z}$ in $\pi_1(X)$. 
Let $f_\gamma$ denote the implementation of $f$ by $\gamma$ with respect to the canonical framing.

Proposition \ref{prop_topological_trivial_case} is then immediately implied by the following result.

\begin{Proposition}
Let $X$ be as in Proposition \ref{prop_topological_trivial_case}, and let $\Sigma$ denote the standard splitting $S^3$ of $X$. Let $\gamma, f_\gamma$ be as above. Then for all integers $k$, the spheres $(f_\gamma)^k(\Sigma)$ are pairwise inequivalent under topological ambient isotopy of $X$ relative to $\partial X$. 
\end{Proposition}

\begin{proof}
Let $\varphi_1:X_1\to S^1\times D^3$ be a homeomorphism. By \cite[Theorem 6.3]{munkres1960obstructions} and \cite{vcernavskiui1969local}, every homeomorphism between smooth closed $3$--manifolds is topologically isotopic to a diffeomorphism. Therefore, we may choose $\varphi_1$ so that it is a diffeomorphism near the boundary. Let $\nu(\partial X_1)$ denote a collar neighborhood of $\partial X_1$ such that $\varphi_1|_{\nu(\partial X_1)}$ is a diffeomorphism. Let $\gamma_1:S^1\to X_1$ be an embedded circle such that its image is contained in $\nu(\partial X_1)$ and such that $\varphi_1\circ \gamma_1(S^1) = S^1\times \{p\}$ for some point $p\in D^3$.
Define $\gamma_2:S^1\to X_2$ similarly. Orient $\gamma_1$, $\gamma_2$ so that $[\gamma_1]+[\gamma_2] = [\gamma]\in H_1(X;\mathbb{Z})$. 

For $i=1,2$, fix a smoothly embedded $D^4$ in $X_i$ that is disjoint from $\nu(\partial X_i)$, and view $X$ as the gluing of $X_1\setminus \mathring{D}^4$ and $X_2\setminus \mathring{D}^4$ with this specific choice of $D^4$'s. 
Let $f_{\gamma_1}$, $f_{\gamma_2}$ denote the implementation of $f$ by $\gamma_1$, $\gamma_2$ respectively, with respect to the canonical framing. We view $f_{\gamma_i}$ as diffeomorphisms supported in $X_i\setminus D^4$. By Remark \ref{rmk_implement_in_S1xD3}, we know that for each $i=1,2$, the extension of $f_{\gamma_i}$ to $X_i$ by the identity is of infinite order in $\pi_0\Homeo_\partial(X_i)/\pi_0\Homeo_\partial(D^4)$.

Now we prove the desired result following the strategy of \cite{tatsuoka2026splitting}. Note that $(f_\gamma)^{k_1}(\Sigma)$ is topologically ambient isotopic to $(f_\gamma)^{k_2}(\Sigma)$ relative to $\partial X$ if and only if $(f_\gamma)^{k_1-k_2}(\Sigma)$ is topologically ambient isotopic to $\Sigma$ relative to $\partial X$. Therefore, we only need to show that if $k\neq 0$, then $(f_\gamma)^k(\Sigma)$ is not topologically ambient isotopic to $\Sigma$ relative to $\partial X$.

Suppose $(f_\gamma)^k(\Sigma)$ is topologically ambient isotopic to $\Sigma$ relative to $\partial X$. Then $(f_\gamma)^k$ is topologically isotopic relative to $\partial X$ to a homeomorphism that preserves the set $\Sigma$. Since $\Homeo^+(S^3)$ is connected, one may isotope so that the homeomorphism fixes $\Sigma$ pointwise. One may further isotope to thicken the fixed point set from $\Sigma$ to a neighborhood of $\Sigma$. As a result, there exist homeomorphisms $g_1, g_2$ of $X_1\setminus \mathring{D}^4$ and $X_2\setminus \mathring{D}^4$ that are the identity near the boundary, such that $(f_\gamma)^k$ is topologically isotopic relative to $\partial X$ to the gluing of $g_1$ and $g_2$. 

We introduce some additional notation.
If $h_1,h_2$ are self-homeomorphisms of $X_1\setminus\mathring{D}^4$ and $X_2\setminus\mathring{D}^4$ that are the identity near the boundary, we use $h_1\# h_2$ to denote the self-homeomorphism of $X$ obtained by gluing $h_1$ and $h_2$. 
If $M$ is a codimension-zero closed submanifold of $N$ and $h$ is a self-homeomorphism of $M$ fixing a neighborhood of the boundary, we use $\mathcal{E}^N_M(h)$ to denote the extension of $h$ from $M$ to $N$ by the identity map. 
If $h_1,h_2$ are two self-homeomorphisms of a compact $4$-manifold $M$, we write $h_1\sim h_2$ if they represent the same element in $\pi_0\Homeo_\partial(M)/\pi_0\Homeo_\partial(D^4)$. 

Then we have
\begin{align}
\mathcal{E}^{X_1\#\hat{X}_2}_{X_1\setminus \mathring{D}^4} \,(g_1) & =  g_1 \# \id_{\hat{X}_2\setminus \mathring{D}^4}\sim  g_1 \# \Big(\mathcal{E}_{X_2\setminus \mathring{D}^4}^{\hat{X}_2\setminus \mathring{D}^4}(g_2)\Big)
\nonumber
\\
& \sim \mathcal{E}_{X_1\#X_2}^{X_1\#\hat{X}_2}\,((f_\gamma)^k) 
 \sim \mathcal{E}_{X_1\setminus D^4}^{X_1\#\hat{X_2}}\,((f_{\gamma_1})^k).
\label{eqn_g1_fgamma1_on_connected_sum}
\end{align}
Here, the first $\sim$ relation holds because $\hat{X}_2\setminus \mathring{D}^4$ is homeomorphic to $D^4$. 
The second $\sim$ relation holds because $(f_\gamma)^k$ is topologically isotopic relative to $\partial X$ to $g_1\# g_2$. 
The last $\sim$ relation holds because $\gamma$ and $\gamma_1$ are homotopic and hence smoothly isotopic in $X_1\# \tilde{X}_2$, and the isotopy preserves the canonical framing.

Since every self-homeomorphism of $\partial D^4$ extends to a self-homeomorphism of $D^4$, there exists a homeomorphism from $X_1\#\hat{X_2}$ to $X_1$ that is the identity on $X_1\setminus \mathring{D}^4$. Therefore, \eqref{eqn_g1_fgamma1_on_connected_sum} implies
\begin{equation}
\label{eqn_g_1_equi_f_gamma1^k}
\mathcal{E}_{X_1\setminus \mathring{D}^4}^{X_1}(g_1) \sim \mathcal{E}_{X_1\setminus \mathring{D}^4}^{X_1}\,((f_{\gamma_1})^k).
\end{equation}
Similarly, 
\begin{equation}
\label{eqn_g_2_equi_f_gamma2^k}
\mathcal{E}_{X_2\setminus \mathring{D}^4}^{X_2}(g_2) \sim \mathcal{E}_{X_2\setminus \mathring{D}^4}^{X_2}\,((f_{\gamma_2})^k).
\end{equation}

Recall that we have a homeomorphism $\varphi_1:X_1\to S^1\times D^3$ such that the image of $\varphi_1\circ \gamma_1$ has the form $S^1\times \{p\}$, and that $\varphi_1$ is a diffeomorphism near the image of $\gamma_1$. 
Therefore, the homeomorphism
\[
\varphi_1\circ \mathcal{E}_{X_1\setminus\mathring{D}^4}^{X_1} (f_{\gamma_1}) \circ \varphi_1^{-1}: S^1\times D^3 \to S^1\times D^3
\]
is smoothly isotopic to the implementation of $f$ by $\varphi_1\circ\gamma_1$. 
By Theorem \ref{thm_Budney_Gabai} Property (2) and Remark \ref{rmk_implement_in_S1xD3}, we know that 
$\varphi_1\circ \mathcal{E}_{X_1\setminus\mathring{D}^4}^{X_1} (f_{\gamma_1})  \circ \varphi_1^{-1}$
is of infinite order in  $\pi_0\Homeo_\partial(S^1\times D^3)/\pi_0\Homeo_{\partial}(D^4)$. Therefore, $\mathcal{E}_{X_1\setminus \mathring{D}^4}^{X_1}\,((f_{\gamma_1})^k)$ is non-trivial in $\pi_0\Homeo_\partial(X_1)/\pi_0\Homeo_{\partial}(D^4)$.
By \eqref{eqn_g_1_equi_f_gamma1^k}, we know that $\mathcal{E}_{X_1\setminus \mathring{D}^4}^{X_1}(g_1)$ is nontrivial in $\pi_0\Homeo_\partial(X_1)/\pi_0\Homeo_{\partial}(D^4)$. 
Similarly,
$\mathcal{E}_{X_2\setminus \mathring{D}^4}^{X_2}(g_2)$ is nontrivial in $\pi_0\Homeo_\partial(X_2)/\pi_0\Homeo_{\partial}(D^4)$.

Let $X^\dagger$ be the normal $m$-fold cover of $X$ such that the preimage of $X_1\setminus \mathring{D}^4$ is connected and the preimage of $X_2\setminus \mathring{D}^4$ has $m$ components. See Figure \ref{figure_cover_connectes_sum_torus} for a schematic figure of $X^\dagger$. 
\begin{figure}
	\begin{overpic}[width=0.35\textwidth]{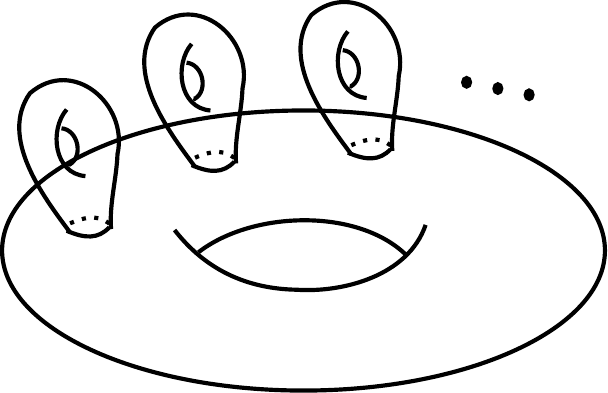}
	\end{overpic}
	\caption{A schematic figure of $X^\dagger$}
    \label{figure_cover_connectes_sum_torus}
\end{figure}
Let $m$ be given by Property (3) of Theorem \ref{thm_Budney_Gabai}, let $(f_{\gamma})^{k\,\dagger}$ be the lifting of $(f_{\gamma})^k$ from $X$ to $X^\dagger$. Note that the preimage of $\gamma$ in $X^\dagger$ is a connected $m$-fold cover of $\gamma$.  Therefore, $(f_{\gamma})^{k\,\dagger}$ is smoothly isotopic to the identity. On the other hand, $X^\dagger$ can be written as a connected sum $X_2\# Y$, where $Y$ is the connected sum of $(m-1)$ copies of $X_2$ and the $m$-fold connected covering of $X_1$. One may attach $2$, $4$-handles to $Y$ so that it becomes a smooth manifold $\hat Y$ homeomorphic to $S^4$. We have
\[
\id|_{X_2\# \hat Y} \sim \mathcal{E}_{X_2\# Y}^{X_2\# \hat Y} ((f_{\gamma})^{k\,\dagger}) 
\sim \mathcal{E}_{X_2 \setminus \mathring{D}^4}^{X_2\# \hat Y} (g_2).
\]
The first $\sim$ relation holds because $(f_{\gamma})^{k\,\dagger}$ is smoothly isotopic to the identity on $X^\dagger = X_2\# Y$ relative to boundary.
The second $\sim$ relation holds because $(f_\gamma)^k$ is topologically isotopic to $g_1\# g_2$ on $X$ relative to $\partial X$, which lifts to the gluing of $g_2$ with a self-homeomorphism of $\hat{Y}\setminus\mathring{D}^4$. Since $\hat{Y} \setminus\mathring{D}^4$ is homeomorphic to $D^4$, the second $\sim$ relation follows.

Since $\hat Y$ is homeomorphic to $S^4$, this implies $\mathcal{E}_{X_2 \setminus \mathring{D}^4}^{X_2} (g_2)$ is trivial in $\pi_0\Homeo_\partial(X_2)/\pi_0\Homeo_{\partial}(D^4)$, which contradicts the earlier result. 
\end{proof}

Proposition \ref{prop_topological_trivial_case} implies the main theorem in the case when $L$, $R$ are topologically trivial in $S_L$, $S_R$:

\begin{proof}[Proof of Theorem \ref{thm_main} when $L$, $R$ are topologically trivial in $S_L$, $S_R$]

    Let $\nu(L)$, $\nu(R)$ be smooth tubular neighborhoods of $L$, $R$. By the uniqueness theorem for topological tubular neighborhoods in dimension $4$ (see \cite[Theorem 5.6]{friedl2025foundations}), we know that $S_L\setminus \nu(L)$ and $S_R\setminus \nu(R)$ are both homeomorphic to $S^1\times D^3$. So the result follows from Proposition \ref{prop_topological_trivial_case}.
\end{proof}

\section{The topologically nontrivial case}
\label{sec_topological_nontrivial}
This section proves Theorem \ref{thm_main} in the case when $L$ is topologically non-trivial in $S_L$ or $R$ is topologically non-trivial in $S_R$. 

We will need the following result due to Daniel Ruberman. It was explained to the authors by Michael Albanese.
\begin{Proposition}[Ruberman]
\label{prop_4mfd_pi2_torsion}
    If $X$ is a $4$-manifold, then $\pi_2(X)$ has no torsion.
\end{Proposition}

\begin{proof}
After removing the boundary of $X$ and taking the universal cover, we may assume without loss of generality that $X$ is simply connected and has no boundary. Hence $\pi_2(X)\cong H_2(X;\mathbb{Z})$. By Poincar\'e duality, 
\[
H_2(X;\mathbb{Z})\cong H^2_c(X;\mathbb{Z}) = \varinjlim H^2(X,X\setminus K;\mathbb{Z}),
\]
where the direct limit takes over all compact subsets $K\subset X$. If a direct limit of abelian groups contains torsion elements, then at least one of the groups contains torsion elements. We show that $H^2(X,X\setminus K;\mathbb{Z})$ has no torsion for all $K$.

Consider the long exact sequence
\[
H^1(X;\mathbb{Z})\to H^1(X\setminus K;\mathbb{Z})\to H^2(X,X\setminus K;\mathbb{Z}) \stackrel{j}{\to} H^2(X;\mathbb{Z}).
\]
By the universal coefficient theorem, we have a split short exact sequence
\[
0\to \Ext(H_{n-1}(X);\mathbb{Z})\to H^n(X;\mathbb{Z})\to \operatorname{Hom}(H_n(X;\mathbb{Z}),\mathbb{Z})\to 0.
\]
Since $X$ is simply connected, we have $H_1(X;\mathbb{Z})=0$, so $H^1(X;\mathbb{Z})=0$, and 
\[
H^2(X;\mathbb{Z})\cong \operatorname{Hom}(H_2(X;\mathbb{Z}),\mathbb{Z}).
\]
Since $\mathbb{Z}$ has no torsion, the group $H^2(X;\mathbb{Z})$ has no torsion. Similarly,
\[
H^1(X\setminus K;\mathbb{Z})\cong \operatorname{Hom}(H_1(X\setminus K;\mathbb{Z}),\mathbb{Z})
\]
has no torsion. Therefore, we have a short exact sequence
\[
0 \to H^1(X\setminus K;\mathbb{Z}) \to H^2(X,X\setminus K;\mathbb{Z}) \to \Ima(j) \to 0,
\]
where both $H^1(X\setminus K;\mathbb{Z})$ and $\Ima(j)$ have no torsion, so $H^2(X,X\setminus K;\mathbb{Z})$ has no torsion.
\end{proof}

We will also need the following lemma. This is a standard result, and similar isomorphisms were used in \cite{budney2019knotted, budney2025automorphism, lin2025mapping}; we give a proof here since we were unable to find a direct reference. 

\begin{Lemma}
\label{lem_wedge_sum_homotopy}
If $A_1,A_2$ are simply connected pointed CW complexes, then 
\[
\pi_2(A_1\vee A_2) \cong \pi_2(A_1)\oplus \pi_2(A_2),
\]
\[
\pi_3(A_1\vee A_2) \cong \pi_3(A_1)\oplus \pi_3(A_2) \oplus (\pi_2(A_1)\otimes \pi_2(A_2)).
\]
Moreover, the isomorphisms from the right-hand side to the left-hand side are given by sums of the maps $\pi_k(A_i)\to \pi_k(A_1\vee A_2)$ induced by inclusions and the Whitehead product 
\[
\Wh: \pi_2(A_1)\otimes \pi_2(A_2) \to \pi_3(A_1\vee A_2).
\]
\end{Lemma}
\begin{proof}
    After replacing $A_1$, $A_2$ by homotopy equivalent CW complexes relative to the base points, we may assume without loss of generality that $A_1$, $A_2$ have no $1$-cells. Then $(A_1\times A_2,A_1\vee A_2)$ is $3$--connected, and hence
    \[
    \pi_2(A_1\vee A_2) \cong \pi_2(A_1\times A_2)\cong \pi_2(A_1)\oplus \pi_2(A_2).
    \]
    This proves the desired isomorphism on $\pi_2$.
    
    To prove the desired isomorphism on $\pi_3$, by the relative Hurewicz theorem and K\"unneth formula, we have 
 \[
\pi_4(A_1\times A_2, A_1\vee A_2)\cong H_4(A_1\times A_2, A_1\vee A_2)
\cong H_2(A_1)\otimes H_2(A_2). 
 \]
 By the  Hurewicz theorem, we have  
 $H_2(A_1)\otimes H_2(A_2)\cong \pi_2(A_1)\otimes \pi_2(A_2).$
Here, all homology groups are in $\mathbb{Z}$ coefficients. 

    Now consider the exact sequence 
    \begin{align*}
       \cdots \to \pi_4(A_1\times A_2)\to \pi_4(A_1\times A_2, A_1\vee A_2) \to \pi_3(A_1\vee A_2) 
       \\
       \to \pi_3(A_1\times A_2) \to \pi_3(A_1\times A_2, A_1\vee A_2) \to \cdots
    \end{align*}
    Since $\pi_k( A_1\vee A_2)\to \pi_k(A_1\times A_2)$ is surjective, we have a short exact sequence
    \[
0 \to \pi_{4}(A_1\times A_2,A_1\vee A_2)\to \pi_3(A_1\vee A_2)\to \pi_3(A_1\times A_2)\to 0,
    \]
    which admits a splitting 
    \[
\pi_3(A_1\times A_2) \cong \pi_3(A_1)\oplus \pi_3(A_2) \to \pi_3(A_1\vee A_2)
    \]
    defined by the sums of the maps $\pi_3(A_i)\to \pi_3(A_1\vee A_2)$ induced by the inclusions. So we have 
    \begin{align*}
\pi_3(A_1\vee A_2) &\cong \pi_3(A_1\times A_2)\oplus \pi_4(A_1\times A_2, A_1\vee A_2) 
\\
&\cong \pi_3(A_1)\oplus \pi_3(A_2)\oplus (\pi_2(A_1)\otimes \pi_2(A_2)).
    \end{align*}
    This proves the desired isomorphism for $\pi_3(A_1\vee A_2)$. The descriptions of the isomorphisms follow from straightforward diagram chasing. 
\end{proof}

The main result of this section is the following proposition, which is a stronger version of Proposition \ref{intro_prop_topological_nontrivial_case}.
\begin{Proposition}
\label{prop_topological_nontrivial_case} 
    Let $X_1$ and $X_2$ be connected, oriented, compact, smooth $4$-manifolds with
nonempty boundary. Suppose that the interior of $X_1$ contains a smoothly
embedded $2$-sphere $S$ with trivial normal bundle whose inclusion into
$X_1$ is not null-homotopic. Suppose also that there exists an element
$a\in\pi_1(X_2)$ such that $a^2\neq 1$. Let
$X=X_1\# X_2$. Then there exists an
infinite family $\{\Sigma_i\}_{i\in\mathbb{N}}$ of smoothly embedded
$3$-spheres in $X$ such that, for each $i$, the closures of the two
components of $X\setminus\Sigma_i$ are diffeomorphic to
$X_1\setminus\mathring{D}^4$ and
$X_2\setminus\mathring{D}^4$, respectively; moreover, the spheres
$\Sigma_i$ are non-homotopic to each other in $X$.
\end{Proposition}

Here, the concept of \emph{homotopy} for embedded submanifolds is defined as follows:
\begin{Definition}
\label{defn_homotopy_S3}
    If $\Sigma_1$, $\Sigma_2$ are two submanifolds of a manifold $M$, we say that $\Sigma_1$ is \emph{homotopic} to $\Sigma_2$ in $M$, if there exists a continuous map $h:\Sigma_1\times [0,1]\to M$ such that $h(x,0) = x$ for all $x\in \Sigma_1$, and $h|_{\Sigma_1\times \{1\}}$ maps $\Sigma_1\times \{1\}$ homeomorphically onto $\Sigma_2$. 
\end{Definition}

We have the following elementary lemma regarding the homotopy of splitting spheres.
\begin{Lemma}
\label{lem_homotopy_S3}
Suppose $X=X_1\# X_2$, where $X_1$ and $X_2$ are oriented connected topological $4$-manifolds with non-empty boundary. Let $\Sigma$ be the standard splitting sphere of $X$, and suppose $f:X\to X$ is a homeomorphism that is the identity on the boundary. Suppose $\Sigma$ is homotopic to $f(\Sigma)$. Then there exists a continuous map $h:\Sigma\times [0,1]\to X$ such that $h(x,0) = x$ and $h(x,1) = f(x)$ for all $x\in \Sigma$.
\end{Lemma}

\begin{proof}
Orient $\Sigma$ so that every arc from $\partial X_1$ to $\partial X_2$ intersects $\Sigma$ with algebraic intersection number $1$. 
Orient $f(\Sigma)$ so that the map $f|_{\Sigma}:\Sigma\to f(\Sigma)$ is orientation-preserving.
Since $f:X\to X$ is a homeomorphism that preserves the boundary, it must be orientation-preserving, and every arc from $\partial X_1$ to $\partial X_2$ intersects $f(\Sigma)$ with algebraic intersection number $1$. Let $[\Sigma], [f(\Sigma)]\in H_3(X;\mathbb{Z})$ be the fundamental classes of $\Sigma$ and $f(\Sigma)$. Then we have 
\[
[f(\Sigma)] = f_*([\Sigma])  \neq -[\Sigma].
\]

Suppose $h:\Sigma\times [0,1]\to X$ is a continuous map such that $h(x,0) = x$ for all $x\in \Sigma$, and $h|_{\Sigma\times \{1\}}$ maps $\Sigma\times \{1\}$ homeomorphically onto $f(\Sigma)$. The degree of $h|_{\Sigma\times \{1\}}$ is either $1$ or $-1$. If the degree of $h|_{\Sigma\times \{1\}}$ is $-1$, we would have 
\[
[f(\Sigma)] = -h_*([\Sigma]) = -[\Sigma],
\]
which yields a contradiction.

Therefore, the map $h|_{\Sigma\times \{1\}}$ is a homeomorphism from $\Sigma$ to $f(\Sigma)$ with degree $1$. Since maps between two $S^3$'s are homotopic if and only if they have the same degree, we may modify $h$ so that $h(x,0) = x$ and $h(x,1) = f(x)$ for all $x\in \Sigma$.
\end{proof}

Now we prove Proposition \ref{prop_topological_nontrivial_case}.
\begin{proof}[Proof of Proposition \ref{prop_topological_nontrivial_case}]
Identify $X$ with the gluing of $X_1\setminus \mathring{D}^4$ and $X_2\setminus \mathring{D}^4$. After isotopy of $S$, we may assume that $S$ is contained in the interior of $X_1\setminus \mathring{D}^4$. 
Let $\Sigma$ be the standard splitting sphere of $X_1\# X_2$.  Take a point $p$ on $\Sigma$, and take an embedded oriented arc $\gamma$ from $p$ to a point on $S$, such that $\gamma$ intersects $S$ and $\Sigma$ transversely at the endpoints, and that the interior of $\gamma$ is disjoint from $S\cup \Sigma$. Let $S',\gamma',p'$ be disjoint parallel copies of $S,\gamma,p$ that satisfy the same properties. Let $\alpha$ be an embedded arc in $X_2\setminus \mathring{D}^4$ from $p$ to $p'$, such that after connecting $p$ and $p'$ by an arc on $\Sigma$, the resulting loop represents an element $[\alpha]\in \pi_1(X_2\setminus \mathring{D}^4,p)\cong \pi_1(X_2,p)$ such that $[\alpha]^2\neq 1$. Note that $[\alpha]$ does not depend on the choice of the arc from $p$ to $p'$ because $\Sigma$ is simply connected. See Figure \ref{figure_parallel_sphere_barbell} for a schematic figure.

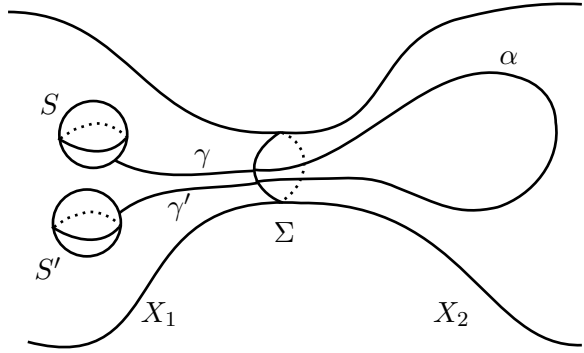
\begin{figure}
	\begin{overpic}[width=0.6\textwidth]{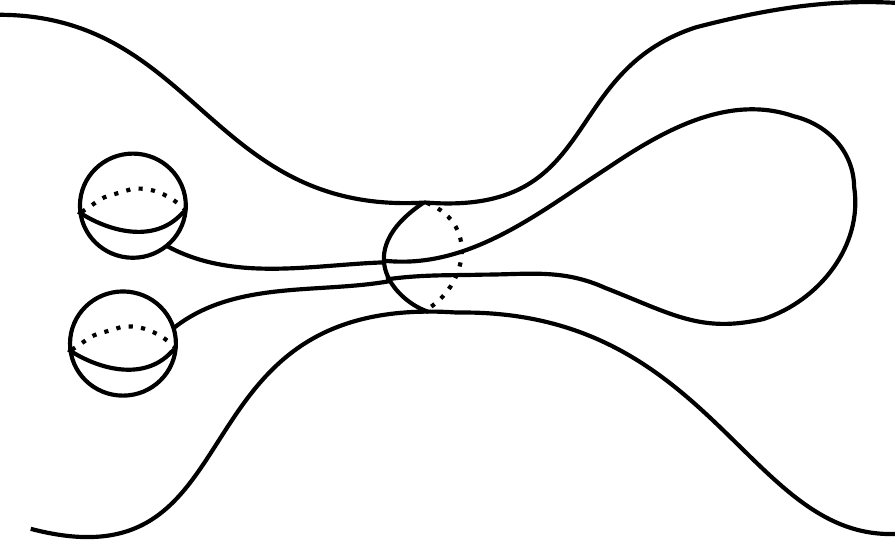}
    \put(50,10){$X_1$}
    \put(160,10){$X_2$}
    \put(100,40){$\Sigma$}
    \put(185,107){$\alpha$}
    \put(70,70){$\gamma$}
    \put(60,50){$\gamma'$}
    \put(12,88){$S$}
    \put(10,26){$S'$}
	\end{overpic}
	\caption{A schematic figure of $S$, $S'$, and the arcs}
    \label{figure_parallel_sphere_barbell}
\end{figure}

The arc $\beta = \gamma^{-1} * \alpha * \gamma'$ is an embedded arc that connects $S$ and $S'$, so $S\cup \beta\cup S'$ defines an implemented barbell diffeomorphism. Denote this diffeomorphism by $f$. We show that for distinct integers $k$, the images $f^k(\Sigma)$ are mutually non-homotopic to each other. For this, we only need to show that $\Sigma$ and $f^k(\Sigma)$ are non-homotopic when $k\neq 0$. 

By the proof of \cite[Proposition 2.8]{budney2025automorphism}, the barbell diffeomorphism is a self-diffeomorphism of $S^2\times D^2\natural S^2\times D^2$ relative to the boundary, such that the homotopy class of the mid-ball is changed by the Whitehead product of the two $S^2$ factors. 

Let $q\in \Sigma$ be a point that is disjoint from the neighborhood of $S\cup \beta\cup S'$ which supports the implemented barbell diffeomorphism $f$, and take $q$ to be the base point of the homotopy groups of $X$. By connecting $S$ to $p$ via $\gamma$ and connecting $p$ to $q$ via an arc on $\Sigma$, the sphere $S$ defines an element in $\pi_2(X_1\setminus\mathring{D^4},q)$. The element does not depend on the choice of the arc connecting $p$ to $q$ because $\Sigma$ is simply connected. Let $[S]\in \pi_2(X_1\setminus\mathring{D^4},q)$ denote this element. 
Since $\pi_2(X_1\setminus\mathring{D^4},q)\cong \pi_2(X_1,q)$, we may view $[S]$ as an element in $\pi_2(X_1,q)$. We will also abuse notation and use $[S]$ to denote its image in $\pi_2(X,q)$. 

Similarly, by connecting $p$ to $q$ with an arc on $\Sigma$, the arc $\alpha$ defines an element in $\pi_1(X_2\setminus\mathring{D}^4,q)$. We abuse notation and use $[\alpha]$ to denote this element as well as the corresponding elements in $\pi_1(X_2,q)$ and $\pi_1(X,q)$. 

Fix a parametrization of $\Sigma$ with base point $q$. Then $f^k(\Sigma)$ defines an element on $\pi_3(X,q)$, which we denote by $[f^k(\Sigma)]$. We have 
\begin{equation}
\label{eqn_fk(S)_homotopy}
[f^k(\Sigma)] = [\Sigma] \pm k \Wh([S],[S]^{[\alpha]}) \pm k \Wh([S],[S]^{([\alpha]^{-1})}),
\end{equation}
where the signs depend on the orientations of $S$ and $S'$. Here, $[S]^{[\alpha]}$ denotes the $\pi_1$-action of $[\alpha]$ on $[S]$, and $[S]^{([\alpha]^{-1})}$ denotes the $\pi_1$-action of $[\alpha]^{-1}$ on $[S]$, and $\Wh(-,-)$ denotes the Whitehead product.

By induction with Lemma \ref{lem_wedge_sum_homotopy} and taking direct limits, we know that if $\{A_i\}_{i\in I}$ is a possibly infinite collection of simply connected pointed CW complexes indexed by a set $I$, then  
\[
\pi_2(\vee_{i\in I} A_i)\cong \oplus_{i\in I} \pi_2(A_i),
\]
\[
\pi_3(\vee_{i\in I} A_i)\cong (\oplus_{i\in I} \pi_3(A_i))\oplus (\oplus_{\{i,j\}\subset I}\pi_2(A_i)\otimes \pi_2(A_j)),
\]
where the isomorphisms from the right-hand sides to the left-hand sides are the sums of maps induced by inclusions and Whitehead products. 

Since both $X_1,X_2$ have non-empty boundary, we have $X=X_1\# X_2 \simeq X_1\vee X_2\vee S^3$. Let $X_1^\dagger$, $X_2^\dagger$ be the universal covers of $X_1$ and $X_2$. Then the universal cover of $X$ is homotopy equivalent to the space 
\[
(\vee_{I_1} X_1^\dagger) \vee (\vee_{I_2} X_2^\dagger) \vee (\vee_{I} S^3),
\]
where 
\begin{align*}
    I_1 &= \pi_1(X_1)\backslash \pi_1(X_1)*\pi_1(X_2), \\
    I_2 &= \pi_1(X_2)\backslash \pi_1(X_1)*\pi_1(X_2), \\
    I &=\pi_1(X_1)*\pi_1(X_2),
\end{align*} 
and the backslashes denote taking right cosets. So we have 
\begin{align*}
\pi_3(X) \cong & (\oplus_{I_1} \pi_3(X_1)) \oplus (\oplus_{I_2} \pi_3(X_2)) \oplus (\oplus_{I} \mathbb{Z}) 
\\
 &\quad \oplus  (\oplus _{\{i,j\}\subset I_1} \pi_2(X_1)\otimes \pi_2(X_1)) \oplus (\oplus _{\{i,j\}\subset I_2} \pi_2(X_2)\otimes \pi_2(X_2)) 
 \\
 &\quad \oplus (\oplus_{i\in I_1,j\in I_2} \pi_2(X_1)\otimes \pi_2(X_2)).
\end{align*}
The element $[\Sigma]$ and its conjugations are in the component $\oplus_{I} \mathbb{Z}$.
The terms $\Wh([S],[S]^{[\alpha]})$ and $\Wh([S],[S]^{[\alpha]})$ both lie in the component 
\[
\oplus _{\{i,j\}\subset I_1} \pi_2(X_1)\otimes \pi_2(X_1).
\]
The term $\Wh([S],[S]^{[\alpha]})$ is in the summand given by $\{i,j\} = \{1,[\alpha]\}$, and $\Wh([S],[S]^{([\alpha]^{-1})})$ is in the summand given by $\{i,j\} = \{1,[\alpha]^{-1}\}$. Since $1\neq [\alpha]^2\in \pi_1(X_2)$, they are in different summands. By Proposition \ref{prop_4mfd_pi2_torsion}, $[S]$ has infinite order in $\pi_2(X_1)$. So by Lemma \ref{lem_tensor_nonvanish} below, $[S]\otimes [S]$ has infinte order in $\pi_2(X_1)\otimes \pi_2(X_1)$. As a result, a non-trivial linear combination of $\Wh([S],[S]^{[\alpha]})$ and $\Wh([S],[S]^{([\alpha]^{-1})})$ cannot equal $[\Sigma]\pm [\Sigma]^{a}$
 for $a\in \pi_1(X)$. By \eqref{eqn_fk(S)_homotopy}, we conclude that $[f^k(\Sigma)]$ is not $\pi_1$-conjugate to $[\Sigma]$ if $k\neq 0$.

By Lemma \ref{lem_homotopy_S3}, if $\Sigma$ and $f^k(\Sigma)$ are homotopic, then $[f^k(\Sigma)]$ and $[\Sigma]$ must be $\pi_1$-conjugate. Therefore, the proposition is proved. 
\end{proof}

Now we present the algebraic lemma that was used in the proof of Proposition \ref{prop_topological_nontrivial_case}.

\begin{Lemma}\label{lem_tensor_nonvanish}
    If $A,B$ are abelian groups, and $a\in A$, $b\in B$ both have infinite order, then $a\otimes b$ has infinite order in $A\otimes B$.
\end{Lemma}

\begin{proof}
    View $A$, $B$ as $\mathbb{Z}$--modules. Since $\mathbb{Q}$ is the localization of $\mathbb{Z}$ by $\mathbb{Z}\setminus\{0\}$, we have $A\otimes_\mathbb{Z} \mathbb{Q}$ is the localization of $A$ by $\mathbb{Z}\setminus\{0\}$. Therefore, if $x\in A$, then $x\otimes 1$ vanishes in $A\otimes_\mathbb{Z}  \mathbb{Q}$ if and only if there exists $n\in \mathbb{Z}\setminus\{0\}$ such that $nx = 0$ in $A$. As a result, the assumption on $a$ implies that $a\otimes 1$ does not vanish in $A\otimes_\mathbb{Z}  \mathbb{Q}$. Similarly, the assumption on $b$ implies that $b\otimes 1$ does not vanish in $B\otimes_\mathbb{Z}  \mathbb{Q}$. Therefore, $(a\otimes 1)\otimes (b\otimes 1)$ does not vanish in $(A\otimes_\mathbb{Z}  \mathbb{Q})\otimes_\mathbb{Q} (B\otimes_\mathbb{Z} \mathbb{Q})$. Since the $\mathbb{Z}$--bilinear map 
    \begin{align*}
    A\times B &\to (A\otimes_\mathbb{Z}  \mathbb{Q})\otimes_\mathbb{Q} (B\otimes_\mathbb{Z}  \mathbb{Q}) \\
     (x,y) &\mapsto (x\otimes 1)\otimes (y\otimes 1)
    \end{align*}
    takes $(a,b)$ to an image with infinite order, we conclude that $a\otimes b$ must have infinite order in $A\otimes B$. 
\end{proof}

Now we finish the proof of the main theorem when one of the components, say $L$, is topologically non-trivial in $S_L$.

\begin{proof}[Proof of Theorem \ref{thm_main} when $L$ is topologically non-trivial in $S_L$]

        Let $\nu(L)$, $\nu(R)$ be smooth tubular neighborhoods of $L$, $R$. Let $L'$ be a parallel copy of $L$ in $S_L\setminus \nu(L)$, we claim that $L'$ must be non-contractible in $S_L\setminus \nu(L)$.  By the unknotting criterion of Swarup \cite{swarup1975unknotting}\footnote{The result of Swarup is stated for a PL embedding of a PL $S^{n-2}$ into a PL $S^{n}$. However, the only place that PL topology is needed is to take the PL tubular neighborhood, so the same result works verbatim for a PL embedding of a PL manifold homeomorphic to $S^{n-2}$ into a PL manifold homeomorphic to $S^n$.}, if $L'$ is contractible in $S_L\setminus \nu(L)$, and if $\pi_1(S_L\setminus \nu(L))$ is accessible, then $S_L\setminus \nu(L)$ is homotopy equivalent to $S^1$. Dunwoody \cite{dunwoody1985accessibility} later proved that all finitely presented groups are accessible. By Freedman's unknotting theorem (see \cite[Theorem 11.7A]{freedman2014topology}), if the complement of a locally flat sphere knot in $S^4$ has fundamental group $\mathbb{Z}$, then the knot is topologically ambient isotopic to the standard embedding. This yields a contradiction to the assumptions. As a consequence, $L'$ is not contractible in $S_L\setminus \nu(L)$. Since $H_2(S_L)=0$, the self-intersection number of $L$ must be zero, so the normal bundle of $L'$ is trivial.

    By Alexander duality, $H_1(S_R\setminus \nu(R)) \cong \mathbb{Z}$. So there exists $a\in \pi_1(S_R\setminus \nu(R))$ that has infinite order.

    Applying Proposition \ref{prop_topological_nontrivial_case} to $X_1 = S_L\setminus \nu(L)$ and $X_2 = S_R\setminus \nu(R)$ then yields the desired result. 
\end{proof}

\bibliographystyle{amsalpha}
\bibliography{references}

\end{document}